\documentclass[10pt]{article}

\usepackage[small]{titlesec}
\usepackage[small,it]{caption}

\usepackage{amsmath}
\usepackage{amsthm}
\usepackage{amssymb}
\usepackage{latexsym}
\usepackage{mathabx}
\usepackage{units} 

\usepackage{ifthen}

\usepackage{tikz}
\usetikzlibrary{arrows, automata, decorations.pathreplacing, fit, matrix, patterns, positioning}
\usepackage{xifthen} 

\usepackage{lineno}

\usepackage{color}
\definecolor{lightgray}{rgb}{0.8, 0.8, 0.8}
\definecolor{darkgray}{rgb}{0.7, 0.7, 0.7}
\definecolor{darkblue}{rgb}{0, 0, .4}

\usepackage[bookmarks]{hyperref}
\hypersetup{
        colorlinks=true,
        linkcolor=darkblue,
        anchorcolor=darkblue,
        citecolor=darkblue,
        urlcolor=darkblue,
        pdfpagemode=UseThumbs,
        pdftitle={todo},
        pdfsubject={Combinatorics},
        pdfauthor={todo},
}

\newcounter{todocounter}

\newtheorem{theorem}{Theorem}
\newtheorem{proposition}[theorem]{Proposition}
\newtheorem{conjecture}[theorem]{Conjecture}

\newenvironment{proof-of-theorem}{\medskip\noindent {\it Proof of Theorem~\ref{thm-gray}.\/}}{\qed\bigskip}

\makeatletter
\newfont{\footsc}{cmcsc10 at 8truept}
\newfont{\footbf}{cmbx10 at 8truept}
\newfont{\footrm}{cmr10 at 10truept}
\renewenvironment{abstract}%
                {
                  \begin{list}{}%
                     {\setlength{\rightmargin}{1in}%
                      \setlength{\leftmargin}{1in}}%
                   \item[]\ignorespaces\begin{small}}%
                 {\end{small}\unskip\end{list}}

\newcommand{\st}{\::\:}

\newcommand{\C}{\mathcal{C}}

\newcommand{\OEISlink}[1]{\href{https://oeis.org/#1}{#1}}
\newcommand{\OEISref}{\href{https://oeis.org/}{OEIS}~\cite{sloane:the-on-line-enc:}}

\newcommand\mybullet{\raisebox{-5pt}{\normalsize \ensuremath{\bullet}}}
\newcommand\mycirc{\raisebox{-5pt}{\normalsize \ensuremath{\circ}}}

\makeatletter
\def\absdot{\@ifnextchar[{\@absdotlabel}{\@absdotnolabel}}
	\def\@absdotlabel[#1]#2{%
		\node at #2 {\normalsize \mybullet};
		\node at #2 [below=2pt] {\ensuremath{#1}};
	}
	\def\@absdotnolabel#1{%
		\node at #1 {\normalsize \mybullet};
	}
\def\absdothollow{\@ifnextchar[{\@absdothollowlabel}{\@absdothollownolabel}}
	\def\@absdothollowlabel[#1]#2{%
		\node at #2 {\normalsize \textcolor{white}{\mybullet}};
		\node at #2 {\normalsize \mycirc};
		\node at #2 [below=2pt] {\ensuremath{#1}};
	}
	\def\@absdothollownolabel#1{%
		\node at #1 {\normalsize \textcolor{white}{\mybullet}};
		\node at #1 {\normalsize \mycirc};
	}
\makeatother

\newcommand{\plotperm}[1]{
	\foreach \j [count=\i] in {#1} {
		\absdot{(\i,\j)};
	};
}

\newcommand{\plotpermbox}[4]{
	\draw [darkgray, thick, line cap=round]
		({#1-0.5}, {#2-0.5}) rectangle ({#3+0.5}, {#4+0.5});
}

\newcommand{\abs}[1]
  {\left\lvert #1 \right\rvert}

\newpagestyle{main}[\small]{
        \headrule
        \sethead[\usepage][][]
        {\sc Universal Layered Permutations}{}{\usepage}}

\title{\sc Universal Layered Permutations}
\author{
	\begin{tabular}{cc}
        Michael Albert&Michael Engen\\
		{\small Department of Computer Science}&{\small Department of Mathematics}\\[-3pt]
		{\small University of Otago}&{\small University of Florida}\\[-3pt]
		{\small Dunedin, New Zealand}&{\small Gainesville, Florida USA}\\[20pt]
        Jay Pantone\footnotemark[\value{footnote}]&Vincent Vatter\footnote{Vatter's research was partially supported by the National Security Agency under Grant Number H98230-16-1-0324. The United States Government is authorized to reproduce and distribute reprints not-withstanding any copyright notation herein.}\\
		{\small Department of Mathematics}&{\small Department of Mathematics}\\[-3pt]
		{\small Dartmouth College}&{\small University of Florida}\\[-3pt]
		{\small Hanover, New Hampshire USA}&{\small Gainesville, Florida USA}\\[20pt]
	\end{tabular}
}

\titleformat{\section}
        {\large\sc}
        {\thesection.}{1em}{}

\date{}

\begin{document}

\maketitle

\pagestyle{main}

\vspace{-0.6in}

\begin{abstract}
We establish an exact formula for the length of the shortest permutation containing all layered permutations of length $n$, proving a conjecture of Gray.
\end{abstract}

We establish the following result, which gives an exact formula for the length of the shortest $n$-universal permutation for the class of layered permutations, verifying a conjecture of Gray~\cite{gray:bounds-on-super:}. Definitions follow the statement.

\begin{theorem}
\label{thm-gray}
For all $n$, the length of the shortest permutation that is $n$-universal for the layered permutations is given by the sequence defined by
\begin{equation}
\tag{$\dagger$}
\label{eqn-thm-gray}
	a(n) = n + \min\{a(k) + a(n-k-1) \st 0 \le k \le n-1\}
\end{equation}
and $a(0)=0$.
\end{theorem}

Up to shifting indices by $1$, the sequence $a(n)$ in Theorem~\ref{thm-gray} is sequence \OEISlink{A001855} in the \OEISref. It seems to have first appeared in Knuth's \emph{The Art of Computer Programming, Volume 3}~\cite[Section 5.3.1, Eq. (3)]{knuth:the-art-of-comp:3}, where it is related to sorting by binary insertion. Knuth shows there that (in our indexing conventions),
\[
	a(n) = (n+1)\lceil\log_2 (n+1)\rceil -2^{\lceil \log_2 (n+1)\rceil} + 1.
\]
This formula also shows that the minimum in \eqref{eqn-thm-gray} is attained when $k=\lfloor n/2\rfloor$. We refer the reader to the OEIS for further information about this old and interesting sequence. 


The permutation $\pi$ of length $n$, thought of in one-line or list notation, is said to \emph{contain} the permutation $\sigma$ of length $k$ if $\pi$ has a subsequence of length $k$ in the same relative order as $\sigma$. Otherwise, $\pi$ \emph{avoids} $\sigma$. For example, $\pi=34918672$ contains $\sigma=51342$, as witnessed by the subsequence $91672$. We frequently associate the permutation $\pi$ with its \emph{plot}, which is the set $\{(i,\pi(i))\}$ of points in the plane.

A set of permutations that is closed downward under the containment order is called a \emph{permutation class}. Thus, $\C$ is a class if for all $\pi$ in $\C$ and all $\sigma$ contained in $\pi$, $\sigma$ is also in $\C$. If $\pi$ has length $m$ and $\sigma$ has length $n$, we define the \emph{sum} of $\pi$ and $\sigma$, denoted $\pi\oplus\sigma$, as the permutation of length $m+n$ defined by
\[
	(\pi\oplus\sigma)(i)
	=
	\left\{\begin{array}{ll}
	\pi(i)&\mbox{for $1\le i\le m$},\\
	\sigma(i-m)+m&\mbox{for $m+1\le i\le m+n$}.
	\end{array}\right.
\]
The plot of $\pi\oplus\sigma$ can therefore be represented as below.
\begin{center}
	$\pi\oplus\sigma=$
	\begin{tikzpicture}[scale=0.5, baseline=(current bounding box.center)]
		\plotpermbox{1}{1}{1}{1};
		\plotpermbox{2}{2}{2}{2};
		\node at (1,1) {$\pi$};
		\node at (2,2) {$\sigma$};
	\end{tikzpicture}
\end{center}

A permutation is \emph{layered} if it can be expressed as a sum of decreasing permutations, which are themselves called \emph{layers}. For example, the permutation $3214657$ is layered, as it can be expressed as $321\oplus 1\oplus 21\oplus 1$ (note that the sum operation is associative and so there is no ambiguity in this expression).

Given a permutation class $\C$, the permutation $\pi$ is said to be \emph{$n$-universal for $\C$} if $\pi$ contains all of the permutations of length $n$ in $\C$ (the alternate term \emph{superpattern} is sometimes used in the literature). A permutation is said to be simply \emph{$n$-universal} if it contains all permutations of length $n$.

To date, the best bounds on the length of the shortest $n$-universal permutation are that it lies between $n^2/e^2$ (a consequence of Stirling's Formula, because if such a permutation has length $m$, the inequality ${m\choose n}\ge n!$ must hold) and ${n+1\choose 2}$, which was established by Miller~\cite{Miller:Asymptotic-boun:}. Even before Miller's result was established, Eriksson, Eriksson, Linusson, and W\"astlund~\cite{eriksson:dense-packing-o:} conjectured that this length is asymptotic to $n^2/2$.

Theorem~\ref{thm-gray} is concerned with $n$-universal permutations for the class of layered permutations. Until this work, the best bounds on the length of the shortest such permutations were obtained by Gray~\cite{gray:bounds-on-super:}, who showed that the length is between $n\ln n-n+2$ and $n\lfloor \log_2 n\rfloor + n$. Similar bounds were obtained independently by Bannister, Cheng, Devanny, and Eppstein~\cite{bannister:superpatterns-a:}. By Theorem~\ref{thm-gray} we see that this length is asymptotic to $n\log_2 n$.

Our first result establishes that among the shortest $n$-universal permutations for the class of layered permutations, there is one that is layered. This is a key component in our proof of Theorem~\ref{thm-gray}.

\begin{proposition}
\label{prop-layered}
Given any permutation $\pi$ of length $m$, there is a layered permutation of length $m$ that contains every layered permutation contained in $\pi$.
\end{proposition}
\begin{proof}
We prove the claim by induction on $m$. Note that the base case is trivial. Let $D$ denote a decreasing subsequence of $\pi$ of maximum possible length. Because $D$ is a maximal decreasing subsequence, every entry of $\pi$ that is not in $D$ must either lie to the southwest of an entry of $D$ or to the northeast of such an entry, but not both. Let $D^-$ denote the set of entries that lie to the southwest of an entry of $D$ and let $D^+$ denote the set of entries that lie to the northeast of such an entry, so that $D$, $D^-$, and $D^+$ together constitute a partition of the entries of $\pi$. An example of this decomposition is shown on the leftmost panel of Figure~\ref{fig-gray}.

Define $\pi^-$ (resp., $\pi^+$) to be the permutation in the same relative order as the entries of $D^-$ (resp., $D^+$), let $\delta = \abs{D}\cdots 21$, and set $\pi^\ast = \pi^-\oplus\delta\oplus\pi^+$. Thus in some sense $\pi^\ast$ is a ``straightened-out'' version of $\pi$; an example is shown in the central panel of Figure~\ref{fig-gray}.

\begin{figure}
\begin{center}
	\begin{tikzpicture}[scale=0.25, baseline=(current bounding box.center)]
		\draw [darkgray, very thick, fill=lightgray] (0.5,11.5) rectangle (4,10);
		\draw [darkgray, very thick, fill=lightgray] (4,10) rectangle (6,9);
		\draw [darkgray, very thick, fill=lightgray] (6,9) rectangle (8,8);
		\draw [darkgray, very thick, fill=lightgray] (8,8) rectangle (9,7);
		\draw [darkgray, very thick, fill=lightgray] (9,7) rectangle (11,2);
		\draw [darkgray, very thick, fill=lightgray] (11,2) rectangle (11.5,0.5);
		\plotperm{3, 5, 4, 10, 1, 9, 6, 8, 7, 11, 2};
		\draw [darkgray, very thick] (0.5,0.5) rectangle (11.5,11.5);
		\node at (5.5,4.5) {$D^-$};
		\node at (10.25,9.25) {$D^+$};
	\end{tikzpicture}
\quad\quad\quad
	\begin{tikzpicture}[scale=0.25, baseline=(current bounding box.center)]
		\plotperm{2,4,3,5,1,10,9,8,7,6,11};
		\draw [darkgray, very thick] (0.5,0.5) rectangle (11.5,11.5);
		\draw [darkgray, very thick] (0.5,0.5) rectangle (5.5,5.5);
		\draw [darkgray, very thick] (5.5,5.5) rectangle (10.5,10.5);
		\draw [darkgray, very thick] (10.5,10.5) rectangle (11.5,11.5);
	\end{tikzpicture}
\quad\quad\quad
	\begin{tikzpicture}[scale=0.25, baseline=(current bounding box.center)]
		\plotperm{1,4,3,2,5,10,9,8,7,6,11};
		\draw [darkgray, very thick] (0.5,0.5) rectangle (11.5,11.5);
		\draw [darkgray, very thick] (0.5,0.5) rectangle (5.5,5.5);
		\draw [darkgray, very thick] (5.5,5.5) rectangle (10.5,10.5);
		\draw [darkgray, very thick] (10.5,10.5) rectangle (11.5,11.5);
	\end{tikzpicture}
\end{center}
\caption{The steps in the proof of Proposition~\ref{prop-layered}. From left to right, the drawings show an example of $\pi$, of $\pi^\ast$, and of the layered permutation $\tau^-\oplus\delta\oplus\tau^+$.}
\label{fig-gray}
\end{figure}
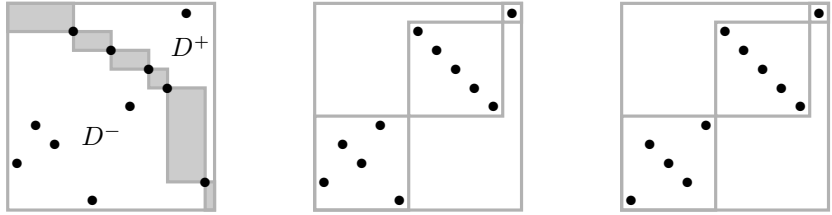

We claim that every layered permutation contained in $\pi$ is also contained in $\pi^\ast$. Suppose $\lambda=\lambda_{1}\oplus\cdots\oplus\lambda_\ell$ is a layered permutation contained in $\pi$, where each $\lambda_{i}$ is a decreasing permutation, and fix an embedding of $\lambda$ into $\pi$. Choose $j$ maximally so that in this embedding of $\lambda$ into $\pi$, the layers $\lambda_{1}\oplus\cdots\oplus\lambda_{j-1}$ are embedded entirely using entries in $D^-$. It follows that the entries of $\lambda_{j+1}\oplus\cdots\oplus\lambda_\ell$ are embedded entirely using entries of $D^+$. Since $\lambda_j$ certainly embeds into $D$ and consequently $\delta$, we have that $\lambda \le \pi^\ast$.

Finally, by induction we see that there are layered permutations $\mu^-$ and $\mu^+$ which contain all of the layered permutations contained in $\pi^-$ and $\pi^+$, respectively. It follows that $\mu^- \oplus \delta \oplus \mu^+$ is layered and contains all of the layered permutations contained in $\pi^\ast$, which in turn contains all of the layered permutations contained in $\pi$, proving the proposition. An example of this final construction is shown in the rightmost panel of Figure~\ref{fig-gray}.
\end{proof}

Another fact which is necessary for the proof of Theorem~\ref{thm-gray} is that when determining whether the layered permutation $\lambda = \lambda_{1}\oplus\cdots\oplus\lambda_\ell$ (where each $\lambda_i$ is a decreasing permutation) is contained in the layered permutation $\pi$, it suffices to take a greedy approach: embed $\lambda_1$ as far to the left as possible, then embed $\lambda_2$ as far left as possible, and so on. We are now ready to prove our main result.

\begin{proof-of-theorem}
We prove the theorem by induction on $n$. As the base case is trivial, suppose that the statement is true for all values less than $n$.

Let $a(n)$ denote the length of the shortest permutation which is $n$-universal for the layered permutations. First we show that
\[
	a(n) \le n + \min\{a(k) + a(n-k-1) \st 0 \le k \le n-1\}.
\]
Choose $k$ so that $a(k) + a(n-k-1)$ is minimized. Then choose a permutation $\sigma$ of length $a(k)$ containing all layered permutations of length $k$ and a permutation $\tau$ of length $a(n-k-1)$ containing all layered permutations of length $n-k-1$. We claim that the permutation $\pi=\sigma\oplus (n\cdots 21)\oplus\tau$ is $n$-universal for the layered permutations. Consider any layered permutation $\lambda=\lambda_1\oplus\cdots\oplus\lambda_\ell$, where each $\lambda_i$ is a decreasing permutation, of length $n=\abs{\lambda_1}+\cdots+\abs{\lambda_\ell}$, and choose $j$ so that
\[
	\begin{array}{lccclclcl}
		\abs{\lambda_1}&+&\cdots&+&\abs{\lambda_{j-1}}&&&\le& k,\\
		\abs{\lambda_1}&+&\cdots&+&\abs{\lambda_{j-1}}&+&\abs{\lambda_j}&>& k.
	\end{array}
\]
Note that this implies that $\abs{\lambda_{j+1}}+\cdots+\abs{\lambda_\ell}\le n-k-1$. Since $\sigma$ is $k$-universal for the layered permutations, it contains $\lambda_1\oplus\cdots\oplus\lambda_{j-1}$. Clearly $\lambda_j$ is contained in $n \cdots 21$. Finally, as $\tau$ is $(n-k-1)$-universal for the layered permutations, it contains $\lambda_{j+1}\oplus\cdots\oplus\lambda_\ell$. This verifies that $\pi$ contains $\lambda$, as desired.

To establish the reverse inequality, suppose that the permutation $\pi$ is $n$-universal for the class of layered permutations and is as short as possible subject to this constraint. Proposition~\ref{prop-layered} allows us to assume that $\pi$ is itself layered, so we may decompose $\pi$ as $\pi_1\oplus\cdots\oplus\pi_\ell$ where each $\pi_i$ is a decreasing permutation.

Because $\pi$ contains all layered permutations of length $n$, it contains $n\cdots 21$. Moreover, as every embedding of $n\cdots 21$ into $\pi$ may use entries from only one layer, there must be some index $j$ such that $\abs{\pi_j} \ge n$. Now choose $k$ so that
\[
	a(k) \le \abs{\pi_1\oplus\cdots\oplus\pi_{j-1}} < a(k+1).
\]
Thus there is at least one layered permutation $\lambda$ of length $k+1$ that does not embed into $\pi_1\oplus\cdots\oplus\pi_{j-1}$. Therefore the earliest that $\lambda$ may embed into $\pi$ is if it embeds into $\pi_1\oplus\cdots\oplus\pi_j$. Now let $\mu$ denote an arbitrary layered permutation of length $n-k-1$. Because $\lambda\oplus\mu$ has length $n$, it must embed into $\pi$, and by our observations about where $\lambda$ may embed into $\pi$ we see that $\mu$ must embed into $\pi_{j+1}\oplus\cdots\oplus\pi_\ell$. As $\mu$ was an arbitrary layered permutation of length $n-k-1$, $\pi_{j+1}\oplus\cdots\oplus\pi_\ell$ is therefore $(n-k-1)$-universal for the class of layered permutations, so we have
\begin{eqnarray*}
	\abs{\pi}
	&=&
	\abs{\pi_{1}\oplus\cdots\oplus\pi_{j-1}}+\abs{\pi_j}+\abs{\pi_{j+1}\oplus\cdots\oplus\pi_\ell},\\
	&\ge&
	a(k)+n+a(n-k-1)\\
	&\ge&
	a(n),
\end{eqnarray*}
completing the proof of the theorem.
\end{proof-of-theorem}

Proposition~\ref{prop-layered} implies that among the shortest permutations which are $n$-universal for the layered permutations, there is at least one that is itself layered. Computation shows that this property is at least slightly unusual. For example, one of the $5$-universal permutations of minimum length for the class of $231$-avoiding permutations is
\[
	1\ 5\ 11\ 9\ 3\ 2\ 8\ 4\ 7\ 6\ 10,
\]
However, this permutation contains $231$ and there is no $231$-avoiding permutation of length $11$ which is $5$-universal for the class of $231$-avoiding permutations. The shortest $231$-avoiding permutations which are $5$-universal for the class of $231$-avoiding permutations instead have length $12$; one such permutation is
\[
	1\ 11\ 3\ 2\ 10\ 7\ 5\ 4\ 6\ 9\ 8\ 12.
\]

On the other hand, computational evidence leads us to suspect that the class of $321$-avoiding permutations does share this property with the class of layered permutations:

\begin{conjecture}
For all $n$, among the shortest permutations that are $n$-universal for the class of $321$-avoiding permutations there is at least one which avoids $321$ itself.
\end{conjecture}

\bibliographystyle{acm}
\bibliography{../../refs}

\end{document}